\documentclass[12pt,twoside]{amsart}
\usepackage{latexsym,amsfonts,amsmath,amssymb,soul}
\usepackage{enumerate,soul}
\usepackage[titletoc]{appendix}
\usepackage[usenames, dvipsnames]{color}
\numberwithin{equation}{section}

\newtheorem{Th}{Theorem}[section]
\newtheorem{Rem}[Th]{Remark}
\newtheorem{Ex}[Th]{Example}
\newtheorem{Lemma}[Th]{Lemma}
\newtheorem{Def}[Th]{Definition}
\newtheorem{Prop}[Th]{Proposition}
\newtheorem{Cor}[Th]{Corollary}

\catcode`\@=11

\renewcommand{\section}%
   {\setcounter{equation}{0}\@startsection {section}{1}{\z@}{-3.5ex plus -1ex
  minus -.2ex}{2.3ex plus .2ex}{\Large\bf}}

\def\CBr{\mathbb R}
\def\CBc{\mathbb C}
\def\CBn{\mathbb N}
\def\CBs{\mathcal S}

\newcommand{\beqsn}{\arraycolsep1.5pt\begin{eqnarray*}}
\newcommand{\eeqsn}{\end{eqnarray*}\arraycolsep5pt}
\newcommand{\beqs}{\arraycolsep1.5pt\begin{eqnarray}}
\newcommand{\eeqs}{\end{eqnarray}\arraycolsep5pt}

\title{About the nuclearity of $\CBs_{(M_{p})}$ and $\CBs_{\omega}$}

\author[Boiti]{Chiara Boiti}
\address{
Dipartimento di Matematica e Informatica \\Universit\`a di Ferrara\\
Via Ma\-chia\-vel\-li n.~30\\
I-44121 Ferrara\\
Italy}
\email{chiara.boiti@unife.it}

\author[Jornet]{David Jornet}
\address{
Instituto Universitario de Matem\'atica Pura y Aplicada IUMPA\\
Universitat Po\-li\-t\`ecni\-ca de Val\`encia\\
Camino de Vera, s/n\\
E-46071 Valencia\\
Spain}
\email{djornet@mat.upv.es}

\author[Oliaro]{Alessandro Oliaro}
\address{Dipartimento di Matematica\\ Universit\`a di Torino\\
 Via Carlo Alberto n. 10\\ I-10123 Torino\\ Italy}
 \email{alessandro.oliaro@unito.it}

\begin{document}



\dedicatory{Dedicated to Prof. Luigi Rodino on the occasion of his 70$^{th}$ 
birthday.}

\maketitle 

\begin{abstract}
We use an isomorphism established by Langenbruch between some sequence spaces and weighted spaces of generalized functions to give sufficient conditions for the (Beurling type) space  $\CBs_{(M_p)}$ to be nuclear. As a consequence, we obtain that for a weight function $\omega$ satisfying the mild condition: $2\omega(t)\leq \omega(Ht)+H$ for some $H>1$
and for all $t\geq0$, the space $\CBs_\omega$ in the sense of Bj\"orck is also nuclear.\end{abstract}


%
\markboth{\sc About the nuclearity of $\CBs_{(M_{p})}$ and $\CBs_{\omega}$}
{\sc C.~Boiti, D.~Jornet and A.~Oliaro}

 

%
%
%
%

\maketitle

\section{Introduction and preliminaries}
\label{CBsec1}

For a sequence $(M_{p})_{p\in\mathbb{N}_{0}}$ which satisfies Komatsu's standard condition $(M2)'$  (stability under differential operators) and, moreover, the condition that there is $H > 0$ such that for any $C >0$ there is $B >0$ with 
\begin{equation}
\label{CB12lan}
s^{s/2} M_p\leq BC^sH^{s+p}
M_{s+p},\qquad\mbox{for any }\ s,p\in\CBn_0,
\end{equation}
where $\CBn_0:=\CBn\cup\{0\}$,
Langenbruch~\cite{CBlan} proves that the Hermite functions are a Schauder basis in the spaces of ultradifferentiable functions of (Beurling type):
\begin{align*}
\CBs_{(M_p)}(\CBr^d):=\Big\{&f\in C^\infty(\CBr^d): \ \mbox{ for any } j\in\CBn,\\
&\sup_{\alpha,\beta\in\CBn_0^d}\sup_{x\in\CBr^d}|x^\alpha D^\beta f(x)|
j^{|\alpha+\beta|}/M_{|\alpha+\beta|}<+\infty\ \Big\}.
\end{align*}
Moreover, in \cite{CBlan} it is also established an isomorphism between $\CBs_{(M_{p})}$ and the K\"othe sequence space:
 \begin{equation*}
\Lambda_{(M_p)}:=\Big\{(c_k)_{k\in\CBn_0}:\ \mbox{ for any }j\in\CBn_0, \ \
\sup_{k\in\CBn_0}|c_k|e^{M(jk^{1/2})}<+\infty\Big\},
\end{equation*}
where 
\begin{equation}
\label{CBmt}
M(t)=\sup_{p}\log\frac {t^pM_0}{M_p},\qquad t>0,
\end{equation}
is  the \emph{associated function} of $(M_p)$.

In this paper we use Grothendieck-Pietsch criterion to characterize when the space $\Lambda_{(M_{p})}$ is nuclear under the assumption that $(M_{p}/M_{0})^{1/p}$ is bounded below by a positive constant and, hence, $M(t)$ is increasing and convex in $\log t$ (see \cite[p.~49]{CBk1}). Indeed, we prove in Theorem~\ref{CBlemma5} that  $\Lambda_{(M_{p})}$ is nuclear if and only if  there is $H>1$ such that for any $t>0$ we have
\begin{equation}
\label{CB4}
M(t)+\log t\leq M(Ht)+H.
\end{equation}
As it is observed in \cite[(2.2)]{CBlan}, condition $(M2)'$ implies \eqref{CB4}. Therefore, conditions $(M2)'$ and \eqref{CB12lan} imply that $S_{(M_{p})}$ is nuclear 
(see Corollary~\ref{CBsmpnuclear}). This should be compared with \cite{CBppv}, where the authors prove that $S_{(M_{p})}$ is nuclear  under Komatsu's conditions $(M1)$ and $(M2).$

As a consequence of Theorem~\ref{CBlemma5} we give a simple proof of the nuclearity of the space $\CBs_\omega$ in the sense of Bj\"orck~\cite{CBbj} given in Definition~\ref{CBsomega}
under the
following condition of Bonet, Meise and Melikhov~\cite{CBbmm} on the weight function $\omega$:
\begin{equation*}
\hspace*{-31mm}\mbox{(BMM)}\qquad\quad
\exists H>1 \ \mbox{s.t.}\quad
2\omega(t)\leq\omega(Ht)+H,\qquad t\geq0.
\end{equation*}
In fact, in this case the space $\CBs_\omega$ is isomorphic to the space $\CBs_{(M_p)}$ for some suitable sequence $(M_{p})$.

\section{Results for the space $S_{(M_{p})}$}
In this section we characterize the nuclearity of $\Lambda_{(M_{p})}$ and give
sufficient conditions for the nuclearity $\CBs_{(M_p)}$. 

We consider a sequence $(M_p)_p$ satisfying the condition that $(M_p/M_0)^{1/p}$
is bounded from below by a positive constant, so that the associated function defined by
\eqref{CBmt} is increasing and convex in $\log t$.

From Grothendieck-Pietsch criterion it is easy to obtain the following
\begin{Lemma}
\label{CBlemma4}
The K\"othe sequence space  $\Lambda_{(M_{p})}$ is nuclear if and only if
for every $j\in\CBn$ there exists $m\in\CBn$ with $m\geq j$ such that
\begin{equation}
\label{CB1}
\sum_{k=0}^{+\infty}e^{M(jk^{1/2})-M(mk^{1/2})}<+\infty.
\end{equation}
\end{Lemma}

\begin{proof}
It follows from Proposition 28.16 of \cite{CBmv}.
\end{proof}

\begin{Th}
\label{CBlemma5}
The space $\Lambda_{(M_{p})}$ is nuclear if and
only if \eqref{CB4} holds.
\end{Th}

\begin{proof}
Let us first remark that \eqref{CB4} implies 
\begin{align*}
M(t)+2\log t&=M(t)+\log t+\log t\\
&\leq M(Ht)+H+\log(Ht)-\log H\\
&\leq M(H^2t)+2H-\log H
\end{align*}
and, more in general,
\begin{equation}
\label{CBlog}
M(t)+N\log t\leq M(H^Nt)+C_{N,H}, \qquad\forall N\in\CBn,
\end{equation}
for some constant $C_{N,H}>0$ depending on $N$ and $H$.

Let us now assume that \eqref{CB4} is satisfied and prove the nuclearity of 
$\Lambda_{(M_{p})}$, using \eqref{CBlog} for a fixed $N>2$.
By Lemma~\ref{CBlemma4},  it's enough to prove  the convergence of the series
\eqref{CB1}.
Indeed, for every
fixed $j\in\CBn$, choosing $m\geq H^Nj$,
\begin{align*}
e^{M(jk^{1/2})-M(mk^{1/2})}&\leq e^{M(jk^{1/2})-M(H^Njk^{1/2})}\\
&\leq e^{M(jk^{1/2})-M(jk^{1/2})-N\log(jk^{1/2})+C_{N,H}}\\
&=e^{C_{N,H}} j^{-N}\frac{1}{k^{N/2}}
\end{align*}
and the series $\sum_{k=1}^{+\infty}\frac{1}{k^{N/2}}$
converges since $N>2$.

Let us now assume that the series \eqref{CB1} converges and prove \eqref{CB4}.
To this aim,
let us first remark that, for $m>j$,
\begin{equation*}
k\longmapsto M(jk^{1/2})-M(mk^{1/2})
\end{equation*}
is decreasing, because $M(e^t)$ is convex by our assumptions (see \cite[p. 49]{CBk1}), 
and therefore its difference quotient $\frac{M(e^t)-M(e^s)}{t-s}$ is increasing
with respect to both variables $t$ and $s$; this implies that
\begin{equation*}
M(mk^{1/2})-M(jk^{1/2})=
\frac{M\left(e^{\log mk^{1/2}}\right)-
M\left(e^{\log jk^{1/2}}\right)}{\log mk^{1/2}-\log jk^{1/2}}
\log\frac mj
\end{equation*}
is increasing with respect to $k$.

Then 
the convergence of \eqref{CB1} implies that
\begin{equation*}
\lim_{k\to+\infty}ke^{M(jk^{1/2})-M(mk^{1/2})}=0
\end{equation*}
and hence
\begin{equation*}
\sup_{k\in\CBn}ke^{M(jk^{1/2})-M(mk^{1/2})}\leq A,
\end{equation*}
for some $A\in\CBr^+$.
Then
\begin{equation*}
\log k+M(jk^{1/2})-M(mk^{1/2})\leq \log A,\qquad \forall k\in\CBn,
\end{equation*}
and hence
\begin{equation}
\label{CBalmost4}
\begin{split}
M(jk^{1/2})-M(mk^{1/2})\leq&-\log k+\log A
=-2\log (jk^{1/2})+\log(j^2 A)\\
\leq&-\log (jk^{1/2})+\log(j^2 A).
\end{split}
\end{equation}

To prove that \eqref{CBalmost4} implies \eqref{CB4} let us first condider $t\geq1$
and choose the smallest $k\in\CBn$ such that $t\leq jk^{1/2}$.
Since
\begin{equation*}
j(k+1)^{1/2}-jk^{1/2}=\frac{j}{\sqrt{k+1}+\sqrt{k}}<j,\qquad\forall k\in\CBn,
\end{equation*}
we have that $jk^{1/2}\in[t,(j+1)t]$ and therefore, from \eqref{CBalmost4},
\begin{align*}
M(t)+\log t&\le M(jk^{1/2})+\log(jk^{1/2})
\le M(mk^{1/2})+\log(j^2A)\\
&=M\left(\frac mj jk^{1/2}\right)+\log(j^2A)\\
&\le M\left(\frac mj (j+1)t\right)+\log(j^2A),
\qquad\forall t\geq1,
\end{align*}
and hence, for $H=\max\left\{\frac mj(j+1), \log(j^2A)+M(1)\right\}$, 
we have that
\eqref{CB4} is satisfied for all $t>0$.
\end{proof}
So, we automatically obtain
\begin{Cor}
\label{CBsmpnuclear}
If $(M_{p})$ satisfies $(M2)'$ and \eqref{CB12lan}, the space $S_{(M_{p})}$ is nuclear.
\end{Cor}

\begin{proof}
The spaces $S_{(M_{p})}$ and $\Lambda_{(M_{p})}$ are isomorphic because $(M_p)$ satisfies $(M2)'$ and \eqref{CB12lan}  by Theorem~3.4 of \cite{CBlan}. Since $(M2)'$ implies \eqref{CB4} (see for instance \cite[(2.2)]{CBlan}), the result follows from Theorem~\ref{CBlemma5}.
\end{proof}

\begin{Rem}
\label{rem3}
\begin{em}
Looking inside the proof of Theorem 3.4 of \cite{CBlan} we can see that in fact
Langenbruch needs only \eqref{CB12lan} and \eqref{CB4}, so that in the above corollary
we could substitute the assumption $(M2)'$ with the condition that the associated function
$M(t)$ satisfies \eqref{CB4}.
\end{em}
\end{Rem}

\section{Results for the space $S_{\omega}.$ Examples}

In this section we  give a sufficient condition for the space $S_{\omega}$ in the sense of Bj\"orck~\cite{CBbj} to be nuclear. 

We
 consider  continuous increasing {\em weight functions} $\omega:\ [0,+\infty)\to
[0,+\infty)$ satisfying:
\begin{itemize}
\item[$(\alpha)$]
$\quad\exists L>0\ $ s.t. $\omega(2t)\leq L(\omega(t)+1),\quad\forall t\geq0$;
\item[$(\beta)$]
$\quad\omega(t)=o(t),\quad$ as $t\to+\infty$;
\item[$(\gamma)$]
$\quad\exists a\in\CBr,\ b>0\ $  s.t. $\omega(t)\geq a+b\log(1+t),
\quad\forall t\geq0$;
\item[$(\delta)$]
$\quad\varphi:\ t\mapsto\omega(e^t)$ is convex.
\end{itemize}
Then we define
$\omega(\zeta):=\omega(|\zeta|)$ for $\zeta\in\CBc^d$.

We denote by $\varphi^*$ the {\em Young conjugate} of $\varphi$, defined by
\begin{equation*}
\varphi^*(s):=\sup_{t\geq0}(ts-\varphi(t)).
\end{equation*}

We recall that $\varphi^*$ is increasing and convex, 
$\varphi^{**}=\varphi$ and  $\varphi^*(s)/s$ is increasing.
Moreover, it will be not restrictive, in the following, to assume $\left.\omega\right|_{[0,1]}\equiv0$ and hence
$\varphi^*(0)=0$.

The space $\CBs_\omega(\CBr^d)$ of weighted rapidly decreasing functions is 
then defined by (see \cite{CBbj}):
\begin{Def}
\label{CBsomega}
$\CBs_\omega(\CBr^d)$ is the set of all $u\in L^1(\CBr^d)$ such that $u,\hat{u}\in
C^\infty(\CBr^d)$ and
\begin{itemize}
\item[(i)]
$\ \forall\lambda>0,\alpha\in\CBn^d_0:\ \sup_{x\in\CBr^d}e^{\lambda
\omega(x)}|D^\alpha u(x)|<+\infty$,
\item[(ii)]
$\ \forall\lambda>0,\alpha\in\CBn^d_0:\ \sup_{\xi\in\CBr^d}e^{\lambda
\omega(\xi)}|D^\alpha \hat{u}(\xi)|<+\infty$,
\end{itemize}
where $D^\alpha=(-i)^{|\alpha|}\partial^\alpha$.
\end{Def}

Note that 
\begin{equation}
\label{CBomega0}
\omega_0(t)=
\begin{cases}
0,&0\leq t\leq 1\cr
\log t, &t>1
\end{cases}
\end{equation}
is a weight function for which $\CBs_{\omega_0}(\CBr^d)$ coincides with
 the classical Schwartz 
class $\CBs(\CBr^d)$.

The space $\CBs_\omega(\CBr^d)$ is a Fr\'echet space 
with different equivalent systems of seminorms (cf. \cite{CBbjo1}, \cite{CBbjo2}, \cite{CBbjo3}).
In particular, we shall use in what follows the family of seminorms
\begin{equation}
\label{CBeffe}
p_\lambda(u)=\sup_{\alpha,\beta\in\CBn_0^d}\sup_{x\in\CBr^d}|x^\beta D^\alpha u(x)|
e^{-\lambda\varphi^*\left(\frac{|\alpha+\beta|}{\lambda}\right)}.
\end{equation}

Given a weight function $\omega$ we construct the sequence $(M_p)$ by
\begin{equation}
\label{CBmp}
M_p=e^{\varphi^*(p)},\qquad\forall p\in\CBn_0.
\end{equation}
Then the associated function of $M_p$ 
is equivalent to the given weight $\omega$.
Indeed, on one side, since $M_0=1$, we have, for $t>0$:
\begin{align*}
M(t)&=\sup_{p\in\CBn_0}\log\frac{t^p}{M_p}=\sup_{p\in\CBn_0}
\left(\log t^p-\log e^{\varphi^*(p)}\right)\\
&\leq\sup_{s\geq0}(s\log t-\varphi^*(s))=\varphi(\log t)=\omega(t).
\end{align*}
On the other side, for $t>0$:
\begin{align*}
\omega(t)&=\sup_{s\geq0}(s\log t-\varphi^*(s))
=\sup_{p\in\CBn_0}\sup_{p\leq s<p+1}(s\log t-\varphi^*(s))\\
&\leq\sup_{p\in\CBn_0}((p+1)\log t-\varphi^*(p))
=\log t+M(t)
\leq 2M(t)+\log M_1
\end{align*}
 since $M(t)\geq \log t -\log M_1$ by definition.

Therefore
\begin{equation}
\label{CBequiv}
M(t)\leq\omega(t)\leq M(t)+\log t\leq
2M(t)+A,
\qquad\forall t>0,
\end{equation}
and for some $A>0$.

Moreover,
\begin{align}
\label{CBvedi7}
\begin{split}
M_p&=e^{\varphi^*(p)}=\exp\{\sup_{t\geq0}(pt-\omega(e^t))\}
=\sup_{t\geq0}\left(e^{pt} e^{-\omega(e^t)}\right)\\
&=\sup_{s\geq1}\left(s^p e^{-\omega(s)}\right)=\sup_{s\geq0}\left(s^p e^{-\omega(s)}\right),
\end{split}
\end{align}
since $\left.\omega\right|_{[0,1]}\equiv0$.

Let us remark that the sequence $(M_p)$  satisfies $(M_p/M_0)^{1/p}\geq1$ and
the condition of
{\em logarithmic convexity}
\begin{equation*}
\hspace*{-50mm}(M1)\qquad\qquad
 M_p^2\leq M_{p-1}M_{p+1},\qquad p\in\CBn,
\end{equation*}
since 
\begin{align*}
2\varphi^*(p)=2\sup_{t\geq0}(tp-\varphi(t))
&\leq\sup_{t\geq0}\big(t(p-1)-\varphi(t)\big)+
\sup_{t\geq0}\big(t(p+1)-\varphi(t)\big)\\
&=\varphi^*(p-1)+\varphi^*(p+1).
\end{align*}

If $\omega$ satisfies condition (BMM), then also $M(t)$ satisfies condition (BMM)
because, by \eqref{CBequiv},
\begin{align}
\label{CBbmmM}
\begin{split}
2M(t)&\leq \frac12 (4\omega(t))\leq\frac12(2\omega(Ht)+2H)\\
&\leq\frac12\omega(H^2t)+\frac32H
\leq M(H^2t)+\frac A2+\frac32H.
\end{split}
\end{align}
Then, by \cite[Prop. 3.6]{CBk1}, we obtain that $(M_p)$ satisfies also the condition of
{\em stability under ultradifferential operators}:
\begin{equation*}
\hspace*{-28mm}(M2)\qquad\quad
\exists A,H>0\  \mbox{s.t.}\qquad M_p\leq AH^p\min_{0\leq q\leq p}M_qM_{p-q}.
\end{equation*}


Moreover, the sequence $(M_p)_p$ satisfies \eqref{CB12lan}.
Indeed, since $\omega(t)=o(t)$ as $t\to+\infty$, we have that for every $\varepsilon>0$
there exists $R_\varepsilon>0$ such that $\omega(t)\leq\varepsilon t+R_\varepsilon$ for all $t\geq0$. Therefore, for $s\geq\varepsilon$,
\begin{equation*}
\varphi^*(s)=\sup_{t\geq0}\left(ts-\omega(e^t)\right)
\geq\sup_{t\geq0}\left(ts-\varepsilon e^t\right)-R_\varepsilon
=s\log \frac s\varepsilon -s-R_\varepsilon,
\end{equation*}
and hence
\begin{equation*}
\left(\frac s\varepsilon\right)^s\leq e^{s+\varphi^*(s)+R_\varepsilon},
\qquad\forall s\geq\varepsilon.
\end{equation*}
Since, for $s\leq\varepsilon$ we have that $s^s\leq(\varepsilon e)^s$, we finally have that for every $s>0$:
\begin{equation*}
s^{s/2}M_p\leq s^se^{\varphi^*(p)}
\leq e^{R_\varepsilon}(\varepsilon e)^s e^{\varphi^*(s)+\varphi^*(p)}
\leq  e^{R_\varepsilon}(\varepsilon e)^se^{\varphi^*(p+s)}=
e^{R_\varepsilon}(\varepsilon e)^s M_{p+s}.
\end{equation*}

If $\omega$ satisfies (BMM), then  $\Lambda_{(M_p)}$
coincides with the sequence space
\begin{equation}
\label{CBlambdaomega}
\Lambda_{\omega}:=\Big\{(c_k)_{k\in\CBn_0}:\ 
\sup_{k\in\CBn_0}|c_k|e^{\omega(jk^{1/2})}<+\infty\ \forall j\in\CBn_0\Big\},
\end{equation}
by \eqref{CBequiv} and \eqref{CBbmmM}.

\begin{Th}
\label{CBteo3}
Let $\omega$ be a weight function. Then $\Lambda_\omega$ is nuclear if and only if
$\omega$ satisfies condition \eqref{CB4}.
\end{Th}

\begin{proof}
As in Theorem~\ref{CBlemma5}, we use \cite[Prop. 28.16]{CBmv} for the sequence 
space $\Lambda_\omega$.
\end{proof}

\begin{Ex}
\label{CBrem9}
\begin{em}
Condition \eqref{CB4} for a weight function $\omega$ is weaker than (BMM).
For instance
\begin{equation*}
\omega(t)=\begin{cases}
0,&0\leq t\leq 1\cr
\log^2 t,& t>1
\end{cases}
\end{equation*}
satisfies  \eqref{CB4} but not (BMM).
\end{em}
\end{Ex}

\begin{Cor}
\label{CBcor1}
Let $\omega$ be a weight function satisfying (BMM). Then $\Lambda_\omega$
is nuclear.
\end{Cor}


\begin{Prop}
\label{CBprop1}
Let $\omega$ be a weight function  satisfying (BMM) and $(M_p)$ the sequence
defined by \eqref{CBmp}. Then
 $\CBs_\omega(\CBr^d)$ is equal (as vector space and as locally convex space) to
$\CBs_{(M_p)}$ and isomorphic to $\Lambda_{(M_p)}=\Lambda_\omega$.
\end{Prop}

\begin{proof}
We endow $\CBs_\omega(\CBr^d)$ with the family of seminorms 
\eqref{CBeffe}. It is isomorphic (and hence equal) to
$\CBs_{(M_p)}$ because, by \cite[formulas (5), (6)]{CBbmm}, 
the following two conditions hold:
\begin{equation*}
\forall j\in\CBn\,\exists \lambda,c>0\ \mbox{s.t.}
\qquad e^{\lambda\varphi^*\left(\frac p\lambda\right)}\leq cj^{-p}M_p,
\qquad\forall p\in\CBn_0,
\end{equation*}
and
\begin{equation*}
\forall \lambda>0\,\exists j\in\CBn,C>0\ \mbox{s.t.}
\qquad j^{-p}M_p\leq Ce^{\lambda\varphi^*\left(\frac p\lambda\right)},
\qquad\forall p\in\CBn_0.
\end{equation*}

Finally, $\CBs_{(M_p)}$  is isomorphic to
$\Lambda_{(M_p)}$ by Theorem~3.4 of \cite{CBlan}, since $(M_p)$ satisfies $(M2)$ (stronger than $(M2)'$)
and \eqref{CB12lan}. Moreover $\Lambda_{(M_p)}$ coincides with $\Lambda_\omega$, as we already remarked in the comment for formula \eqref{CBlambdaomega}.
\end{proof}

Condition \eqref{CB4}, written in terms of the weight function $\omega$, is equivalent to the nuclearity of $\Lambda_\omega$ by 
Theorem~\ref{CBteo3}, but it is not necessary for the nuclearity of $\CBs_\omega$.
For example, the weight $\omega_0(t)$ defined by \eqref{CBomega0} does not
satisfy \eqref{CB4} and hence $\Lambda_{\omega_0}$ is not nuclear, while
$\CBs$ is well known to be nuclear.
In particular, $\Lambda_{\omega_0}$ and $\CBs$ are not isomorphic. On the other hand, from the results that we have we do not know if condition \eqref{CB4} is sufficient for the nuclearity of $\CBs_{\omega}$, but we need the stronger condition (BMM), as we state in the following

\begin{Th}
\label{CBthnuclear}
Let $\omega$ be a weight function satisfying (BMM). Then $\CBs_\omega$ is
a nuclear space.
\end{Th}

\begin{proof}
It follows from Proposition~\ref{CBprop1} and Corollary~\ref{CBcor1}.
\end{proof}

%

\begin{Ex}
\begin{em}
There exist sequences $(M_p)_p$ satisfying  \eqref{CB12lan} and \eqref{CB4} (for 
the associated function), but not (M2).

Let us consider, for example, a weight function $\omega$ satisfying \eqref{CB4} but not
(BMM) (see Example~\ref{CBrem9}) and construct the sequence $(M_p)$ as in
\eqref{CBmp}. Then $M_p$ satisfies (M1),  \eqref{CB12lan}  and its associated function
satisfies \eqref{CB4} because, by \eqref{CBequiv} and \eqref{CBlog}:
\begin{align*}
M(t)+\log t &\leq \omega(t)+2\log t-\log t
\leq \omega(H^2t)+C_{2,H}-\log t\\
&\leq M(H^2t)+\log(H^2t)+C_{2,H}-\log t
=M(H^2t)+2\log H+C_{2,H}.
\end{align*}
However, $M(t)$ does not satisfy (BMM) because $\omega(t)$ does not satisfy (BMM)
(see \eqref{CBequiv}), therefore $M_p$ does not satisfy (M2) by
\cite[Prop. 3.6]{CBk1}.
\end{em}
\end{Ex}

The example above furnishes a sequence $(M_p)$ satisfying (M1), but not (M2), for which the space $\CBs_{(M_p)}$ is nuclear, by Corollary~\ref{CBsmpnuclear} and Remark~\ref{rem3}. Comparing with \cite{CBppv} it is then interesting
the following
\begin{Cor}
Condition (M2) is not necessary for the nuclearity of $S_{(M_p)}$.
\end{Cor}

\vspace*{10mm}
{\bf Acknowledgments.}
The authors were partially supported by the Projects FAR 2017, FAR 2018 and FIR 2018 (University of Ferrara), FFABR 2017 (MIUR). 
The research of the second author was partially supported by the project MTM2016-76647-P.
The first and third authors are members of the Gruppo Nazionale per l'Analisi Matematica, la Probabilit\`a e le loro Applicazioni (GNAMPA) 
of the Istituto Nazionale di Alta Matematica (INdAM).


\end{document}